\documentclass[a4paper, 11pt]{amsart} 

\usepackage[letterpaper,margin=1in]{geometry}
\usepackage{etex,float}
\usepackage{amsmath,amssymb,amsthm,amsfonts,mathrsfs, mathtools}
\usepackage[frame,cmtip,arrow,matrix,line,graph,curve]{xy}
\usepackage{graphpap,color,paralist,pstricks}
\usepackage[mathscr]{eucal}
\usepackage{mathabx}
\usepackage[pdftex,colorlinks,backref=page,citecolor=blue]{hyperref}
\usepackage{tikz}
\usetikzlibrary{calc,decorations.markings}
\usepackage{epic,eepic}
\usepackage{yfonts}
\usepackage{enumitem} 
\usepackage{bbm}
\usepackage{tikz-cd}
\usepackage{aliascnt}
\usepackage{mathbbol}

\allowdisplaybreaks

\usepackage{xcolor, color, soul}

\usepackage{color}

\newtheorem{headthm}{Theorem}

\newaliascnt{headcor}{headthm}

\aliascntresetthe{headcor}

\newaliascnt{headconj}{headthm}

\aliascntresetthe{headconj}

\newaliascnt{corollary}{theorem}
\newtheorem{corollary}[corollary]{Corollary}
\aliascntresetthe{corollary}

\newaliascnt{claim}{theorem}

\aliascntresetthe{claim}

\newaliascnt{lemma}{theorem}
\newtheorem{lemma}[lemma]{Lemma}
\aliascntresetthe{lemma}

\newaliascnt{thmdfn}{theorem}

\aliascntresetthe{thmdfn}

\newaliascnt{conjecture}{theorem}

\aliascntresetthe{conjecture}

\newaliascnt{proposition}{theorem}

\aliascntresetthe{proposition}

\theoremstyle{definition}
\newaliascnt{definition}{theorem}

\aliascntresetthe{definition}

\newaliascnt{notation}{theorem}

\aliascntresetthe{notation}

\newaliascnt{example}{theorem}

\aliascntresetthe{example}

\newaliascnt{examples}{theorem}

\aliascntresetthe{examples}

\newaliascnt{remark}{theorem}

\aliascntresetthe{remark}

\newaliascnt{fact}{theorem}

\aliascntresetthe{fact}

\newaliascnt{question}{theorem}

\aliascntresetthe{question}

\newaliascnt{questions}{theorem}

\aliascntresetthe{questions}

\newaliascnt{problem}{theorem}

\aliascntresetthe{problem}

\newaliascnt{construction}{theorem}

\aliascntresetthe{construction}

\newaliascnt{setup}{theorem}

\aliascntresetthe{setup}

\newaliascnt{algorithm}{theorem}

\aliascntresetthe{algorithm}

\newaliascnt{observation}{theorem}

\aliascntresetthe{observation}

\newaliascnt{discussion}{theorem}

\aliascntresetthe{discussion}

\newaliascnt{defprop}{theorem}

\aliascntresetthe{defprop}

\def\sectionautorefname~#1\null{Section #1\null}
\def\subsectionautorefname~#1\null{\S #1\null}

\makeatletter
\def\equationautorefname~#1\null{Equation\,(#1)\null}
\makeatother

\definecolor{myorange}{RGB}{255, 160, 70}
\definecolor{darkgreen}{RGB}{0, 166, 0}

\def \height{{\operatorname{ht}}}

\def\Ass{\operatorname{Ass}}
\def\Min{\operatorname{Min}}

\def \trdeg{{\operatorname{trdeg}}}

\def \ord{{\operatorname{ord}}}

\DeclareMathOperator{\Quot}{Quot}

\DeclareMathOperator{\sell}{s\ell}
\DeclareMathOperator{\dell}{d\ell}

\def \0{\mathbf 0}

\def \f1{\mathbf{1}}

\newcommand{\kk}{\mathbb{k}}

\def\xi{x}

\def\ls{\leqslant}
\def\gs{\geqslant}

\def\fm{\mathfrak{m}}
\def\frp{\mathfrak{p}}
\def\frq{\mathfrak{q}}

\def\fq{\mathbf{q}}

\def\fq{\mathbf{q}}

\def \RR{\mathbb R}

\def \ZZ{\mathbb Z}

\def \I{\mathcal I}

\begin{document}
	
	\title[The symbolic and divisorial analytic spreads are finite]{The symbolic and divisorial analytic spreads are finite}

	\author[Jonathan Monta{\~n}o]{Jonathan Monta{\~n}o$^*$}
	\address{Jonathan Monta{\~n}o\\School of Mathematical and Statistical Sciences, Arizona State University, P.O. Box 871804, Tempe, AZ 85287-1804, USA. \emph{Email:} {\rm montano@asu.edu}}
	\thanks{$^{*}$ The  author was partially funded by NSF Grant DMS \#2401522.}

	\begin{abstract}
Let $R$ be a  ring that is essentially of finite type over a field and $I\subseteq R$ be an ideal. In this article it is shown that the minimal number of generators of the symbolic powers $I^{(n)}$ is bounded above by a polynomial in $n$. Furthermore, if  $R$ is assumed to be a domain and $\mathcal I=\{I_n\}_{n\ge 0}$ is an  $\mathbb{R}$-divisorial filtration,   then the minimal number of generators of the ideals $I_n$ is again bounded above by a polynomial in $n$. 
	\end{abstract}

		\maketitle

		\begin{center}
	{\bf Disclaimer}
\end{center}

The proofs of the main results  in this paper were generated by ChatGPT Pro.   The author wrote the manuscript, organized and presented the arguments, and carefully checked the correctness of each proof.
	
\section{Introduction}\label{s:the_intro}

Let $R$ be a Noetherian local ring 
and let $I\subseteq R$ be an ideal. The \emph{symbolic powers of $I$} are defined by
\begin{equation}\label{eq:sym_powers}
	I^{(n)}
	=
	\bigcap_{\mathfrak p\in\Min(I)}
	\bigl(I^nR_{\mathfrak p}\cap R\bigr),
\end{equation}
where $\Min(I)$ denotes the set of minimal primes of $I$. The construction of symbolic powers has generated a vast amount of research in both commutative algebra and algebraic geometry; see \cite{SymbolicPowersSurvey18,GrifoSeceleanuSurvey21} for surveys on the subject. 
Despite the extensive literature on symbolic powers, however, the following natural question surprisingly remains open: \emph{is the sequence of numbers of generators
$
\{\mu(I^{(n)})\}_{n\gs 0}
$
bounded above by a polynomial in $n$?} Recall that the analogous question for ordinary powers has a positive answer. Indeed, one has
$
\mu(I^n)=O(n^{\ell(I)-1})
$
, where $\ell(I)$ denotes the \emph{analytic spread of $I$}.

The analogous notion for symbolic powers, the
\emph{symbolic analytic spread}, is  defined as
\[
\sell(I)
=
\inf\bigl\{t\in\RR\mid \mu(I^{(n)})=O(n^{t-1})\bigr\}.
\]
One notes that $\sell(I)$ is not known to be finite in general except in special cases, including: when the symbolic Rees algebra
$
\oplus_{n\gs 0}I^{(n)}
$ 
is Noetherian, by standard methods on Hilbert functions; under mild hypotheses on $R$ and $I$ when $\dim(R/I)\ls 2$, by work of Dutta \cite{Dutta83}, and when $\dim(R/I)\ls 3$, by work of the author and Dao \cite{DaoMontano21}; and when the rings $R/I^{(n)}$ have sufficiently high depth for $n\gg 0$, by the same work of the author and Dao. Related finiteness results on $\sell(I)$  under additional hypotheses were also
obtained by Chen in 
\cite[Theorem~4.2.8 and Corollary~4.2.11]{Chen19}.

The first main theorem of this paper shows $\sell(I)<\infty$ when $R$ is essentially of finite type over a field, with no further assumptions on $I$.

\begin{headthm}\label{thm:symbolic-main}
	Let $R$ be a ring that is essentially of finite type over a field, and let $I\subseteq R$ be an ideal. Then
	$
	\sell(I)<\infty
	$.
\end{headthm}

We make a few remarks about \autoref{thm:symbolic-main}. First, although $R$ was initially assumed to be local, the theorem does not require this hypothesis. This is because the proof does not use this assumption, and we may define $\mu(M)$ to be the smallest cardinality of a generating set of the $R$-module $M$.  We also note that some authors define symbolic powers by taking the intersection in \autoref{eq:sym_powers} over $\Ass(I)$, the associated primes of $I$, rather than $\Min(I)$. The proof of the theorem applies equally to this alternative definition. 

Combining \autoref{thm:symbolic-main} with \cite[Proposition 3.15]{DaoMontano21}, we conclude that if in addition $R$ is a normal local domain of characteristic $p>0$, then the \emph{Frobenius complexity} of $R$ is finite. This gives an affirmative answer to the corresponding question in \cite[Question 4.11]{EnescuYao16} for normal local domains that are essentially of finite type over a field.

We now turn to divisorial filtrations. Assume that $R$ is a domain. Let $v$ be a valuation of $\Quot(R)$, the quotient field of $R$. Let  $(V_v,\fm_v,\kk_v)$ be the  valuation ring of $v$, and assume that $R\subseteq V_v$. Set
$
\frp=\fm_v\cap R
$, 
the center of $v$ on $R$. We say that $v$ is a \emph{divisorial valuation of $R$} if 
$
\trdeg_{\kk(\frp)}(\kk_v)=\height(\frp)-1.
$

A sequence of ideals $\I=\{I_n\}_{n\gs 0}$ is called a \emph{filtration} if $I_0=R$,
$
I_nI_m\subseteq I_{n+m}$ 
 for all $n,m\gs 0$, 
and
$
I_{n+1}\subseteq I_n$
 for all $n\gs 0$. 
A filtration is called \emph{$\RR$-divisorial} if there exist divisorial valuations $v_1,\ldots,v_s$ and positive real numbers $\lambda_1,\ldots,\lambda_s$ such that, for every $n\gs 0$,
\[
I_n
=
\bigcap_{i=1}^s I(v_i)_{n\lambda_i},
\quad \text{
where},\quad \text{ for }\alpha\in\RR_{\gs 0},\quad 
I(v_i)_\alpha
=
\{f\in R\mid v_i(f)\gs\alpha\}
.
\] 
Divisorial filtrations have been shown to exhibit good behavior in a variety of settings; see for example  \cite{CutkMontano2,Cutk21,Cutk25,Cutk26}. Important examples include symbolic powers of prime ideals in regular rings and the integral closure of powers of an ideal $
\{\overline{I^n}\}_{n\gs 0}$.

As in the symbolic-power setting, for an $\RR$-divisorial filtration $\I=\{I_n\}_{n\gs 0}$, we define the \emph{divisorial analytic spread} by
\[
\dell(\I)
=
\inf\bigl\{t\in\RR\mid \mu(I_n)=O(n^{t-1})\bigr\}.
\]
We note that, when $R$ is local, $\dell(\I)$ is different from the analytic spread
$\ell(\I)$ of a filtration introduced  in
\cite{cutkosky2022analyticsym}, which is defined as the dimension of
the fiber cone.

In \cite{cutkosky2023analytic} Cutkosky proved that if $R$ is   a two-dimensional normal excellent local ring and the $\lambda_i$ are rational, then   $\dell(\I)\ls 2$. 
Besides this result, we are not aware of other bounds for
$\dell(\I)$ that apply to general classes of divisorial filtrations.

The second main theorem of this paper shows $\dell(\I)<\infty$ when $R$ is a domain that is essentially of finite type over a field. 

\begin{headthm}\label{thm:div-main}
	Let $R$ be a domain that is essentially of finite type over a field, and let
	$\I=\{I_n\}_{n\gs 0}$ be an $\RR$-divisorial filtration. Then 
$
	\dell(\I)<\infty
$. 
\end{headthm}

We conclude by noting that the existence of a polynomial bound for the number of generators does not simply follow by the definition of filtration. This can be seen from \cite[Example 1.1]{Cutk25}, which was constructed by the author in joint work with Dao and Cutkosky.

As a natural follow-up problem, one may want to find effective bounds on $\sell(I)$ and $\dell(\I)$. For example, does $\sell(I)\ls \dim(R)$ always hold? See \cite[Conjecture 5.1]{DaoMontano21}. It would also be interesting to extend the finiteness results to more general Noetherian rings.

	
		\section{Preliminary lemmas}	
		
		In this section we include some preliminary lemmas that are used in the proofs of the two main results. 
		
		We begin with the following lemma that helps us express symbolic powers as saturations by a single element. The special cases corresponding to 
		$\{\frp_1,\dots,\frp_r\}=\Min(I)$ and
		$\{\frp_1,\dots,\frp_r\}=\Ass(I)$ are  stated in
		\cite[Lemmas~2.1 and~2.2]{HJKN2023}. Related saturation descriptions
		appear in \cite[Proposition~3.13]{Eisenbud}, 
		\cite[Section~3]{herzog2007}, and \cite[Propositions~2.7.30-2.7.33]{Chen19}.

		\begin{lemma}\label{lem:uniform-saturation}
			Let $R$ be a Noetherian ring, and let $I\subseteq R$ be an ideal. Let
			$\{\frp_1,\dots,\frp_r\}$ 
			be a fixed finite set of prime ideals of $R$. 
			Set $W=R\setminus\bigcup_{j=1}^r\frp_j$. Then there exists an element $h\in W$ such that for every $n\gs 1$ 
			\[
		\bigcap_{j=1}^r\bigl(I^nR_{\frp_j}\cap R\bigr)=
			I^n:h^\infty.
			\]
		\end{lemma}
		
		\begin{proof}
		By Brodmann's theorem \cite{Brodmann79}, the set
			$
			\Ass^\infty(I) =\bigcup_{n\ge 1}\Ass(I^n)
		   $
			is finite. Define
			\[
			\mathcal A=\{\frp\in\Ass^\infty(I)\mid \frp\cap W\ne\varnothing\}.
			\]
			For each $\frp\in\mathcal A$, choose $h_\frp\in \frp\cap W$, and let
		    $
			h=\prod_{\frp\in\mathcal A}h_\frp,
			$
			with the convention $h=1$ if $\mathcal A=\varnothing$. Since $W$ is multiplicatively closed, for every $\frp\in \Ass^\infty(I)$ one has 
			$
			h\in \frp$ if and only if $\frp\in \mathcal A$.
			
			Fix $n$ and choose a minimal primary decomposition
			$
			I^n=\bigcap_{j=1}^t \fq_j$. Therefore, 
			\[
			I^n:h^\infty = \bigcap_{j=1}^t \bigl(\fq_j:h^\infty\bigr) = \bigcap_{j=1}^t \tilde{\fq_j},
			\]
			where $\tilde{\fq_j} = \fq_j$ if $h\not\in \sqrt{\fq_j}$, and $\tilde{\fq_j}=R$ otherwise. 
			By the last observation in the previous paragraph we also have $h\not\in \sqrt{\fq_j}$ if and only if $\sqrt{\fq_j}\cap W =\varnothing$, which is equivalent to $\sqrt{\fq_j}\subset \bigcup_{i=1}^r\frp_i$. 
			 By prime avoidance the latter holds if and only if $\sqrt{\fq_j}\subset \frp_i$ for some  $i$, and hence the conclusion follows. 
		\end{proof}
		
		\begin{corollary}\label{cor:symbolic-saturation}
			For every ideal $I$ in a Noetherian ring $R$, there exists $h\in R$ such that
			\[
			I^{(n)}=I^n:h^\infty
			\]
			for every $n\ge 1$.
		\end{corollary}
		
		\begin{proof}
			Apply \autoref{lem:uniform-saturation} with either $\{\frp_1,\dots,\frp_r\} = \Ass(I)$, or $\{\frp_1,\dots,\frp_r\} = \Min(I)$, depending on the chosen definition of symbolic powers. 
		\end{proof}
		
		The following lemma realizes saturation by a single element as an ideal obtained by elimination from a polynomial extension. This construction is commonly known as the ``Rabinowitsch trick''.
		
		\begin{lemma}\label{lem:saturation-elimination}
			Let $R$ be a commutative ring, $I\subseteq R$ an ideal, and $h\in R$. If $z$ is a variable, then
			\[
			I:h^\infty=\bigl(IR[z] + (1-zh)\bigr)\cap R.
			\]
		\end{lemma}
		
		\begin{proof}
			The ideal $\bigl(IR[z] + (1-zh)\bigr)\cap R$ is the kernel of the natural map 
			\[
			  R\longrightarrow R[z]/\bigl(IR[z] + (1-zh)\bigr)\cong (R/I)_h.	
			  \]
			This kernel consists precisely of the elements $f\in R$ for which $h^nf\in I$ for  $n\gg 0$, namely $I:h^\infty$.
		\end{proof}
	
	We now give a lemma that allows us to express an intersection of contractions of ideals from different polynomial extensions as a single contraction
		
		\begin{lemma}\label{lem:intersection-elimination}
			Let $R$ be a commutative ring. For $1\le i\le s$, let $T_i=R[\mathbf x_i]=R[x_{i,1},\ldots, x_{i,r_i}]$ be polynomial rings in pairwise disjoint sets of variables and let $L_i\subseteq T_i$ be ideals. Let $z_1,\ldots, z_s$ be new variables and write 
			$
			T=R[\mathbf x_1,\dots,\mathbf x_s,z_1,\dots,z_s]$.  
			Define
			\[
			D=\sum_{i=1}^s z_iL_iT+\Bigl(1-\sum_{i=1}^sz_i\Bigr)T.
			\]
			Then
			\[
			D\cap R=\bigcap_{i=1}^s(L_i\cap R).
			\]
		\end{lemma}
		
		\begin{proof}
			Suppose $f\in\bigcap_i(L_i\cap R)$. Then
			\[
			f=\sum_{i=1}^sz_if+\Bigl(1-\sum_{i=1}^sz_i\Bigr)f\in D. 
			\]
			Therefore,  $f\in D\cap R$.
			
			Conversely, suppose $f\in D\cap R$. Fix $i$ and consider the ideal $J_i=(z_1,\ldots, z_{i-1},z_i-1,z_{i+1}, \ldots, z_s)$. The quotient map  $T\to T/J_i\cong R[\mathbf x_1,\dots,\mathbf x_s]$ maps $D$ onto the extended ideal $L_iR[\mathbf x_1,\dots,\mathbf x_s]$. Hence $f$ belongs to the latter ideal. Since polynomial extensions are faithfully flat, the  contraction back to $T_i$ gives $f\in L_i$.  Thus, $f\in L_i\cap R$. As this holds for every $i$, we obtain $f\in\bigcap_i(L_i\cap R)$, finishing the proof.
		\end{proof}

The following lemma shows that a polynomial bound on the degrees of generators gives a polynomial bound after elimination. As a consequence, the number of generators of the resulting elimination ideals is also polynomially bounded.
	
	\begin{lemma}\label{lem:effective-elimination}
		Let 
		\[
		A=\kk[x_1,\dots,x_t]\subseteq T=A[y_1,\dots,y_r]
		\]
		be  polynomial rings over a field $\kk$. Let $\{L_n\}_{n\gs 1}$ be a sequence of ideals $L_n\subseteq T$ and suppose that  $L_n$ has a generating set consisting of polynomials of total degree at most
		$
		d_n\le an^b
		$
		for some constants $a,b\in \ZZ_{>0}$. Then there is a constant $c\in\mathbb Z_{>0}$ such that $L_n\cap A$ is generated by polynomials of total  degree at most
		$
		cn^{b2^{t+r}}
		$ 
		for every $n$. Consequently, 
		\[
		\mu(L_n\cap A) = O(n^{tb2^{t+r}}).
		\]
	\end{lemma}
	
	\begin{proof}
		 Choose an elimination order for the $y$-variables on  $T$,  with the $y$-variables larger than the $x$-variables. Dub\'e's theorem \cite{Duble90} gives a Gr\"obner basis $G_n$ of $L_n$ whose elements have total degree at most
		\[
		2\Bigl(\frac{d_n^2}{2}+d_n\Bigr)^{2^{t+r-1}}
		\le cn^{b2^{t+r}}
		\]
		for a constant $c$ independent of $n$. By the elimination theorem \cite[Ch.3, \S1]{CoxLigtOshea25}, $G_n\cap A$ is a Gr\"obner basis, hence a generating set, of $L_n\cap A$. This proves the first part of the statement.
		
		Set $e_n=cn^{b2^{t+r}}$. The vector space of polynomials in $A$ of total degree at most $e_n$ has dimension 
	    $\binom{e_n+t}{t}$. 
		By replacing $G_n\cap A$ by a $\kk$-basis of its $\kk$-linear span, we obtain a generating set of $L_n\cap A$ with at most $\binom{e_n+t}{t}$ elements. Thus, $\mu(L_n\cap A) = O(n^{tb2^{t+r}})$, as desired.
		\end{proof}
 
We end this section with the following lemma, which shows that divisorial valuation rings over domains essentially of finite type over a field are  essentially of finite type over the same field.

\begin{lemma}\label{lem:affine-dvr-model}
Let $R$ be a domain that is essentially of finite type over a field $\kk$, and let $\Quot(R)$ be its quotient field. 
		Write  
	$
	R=W^{-1}R_0$
	with  $ R_0$  
	 a domain that is finitely generated over $\kk$ and $W\subseteq R_0$ a multiplicatively closed set.  
	Let $v$ be a divisorial valuation of $R$ and $V\subset \Quot(R)$ its valuation ring. Then there exist a domain
	$
	R_0\subseteq B\subseteq \Quot(R)
	$ 
	finitely generated as an $R_0$-algebra, 
	and a height-one prime ideal $\frp\subset B$ such that
	$
	B_\frp=V
	$.
\end{lemma}

\begin{proof}
	Let $\fm_V$ be the maximal ideal of $V$. 
 	By   \cite[Theorem~9.3.2 and Proposition~6.6.2]{HunekeSwanson} $V$ is a DVR and essentially of finite type over $R$. Thus,  there is a finitely generated $R$-subalgebra
	$
	C=R[b_1,\dots,b_r]\subseteq V
	$ 
	such that  $C_{\frq}=V$, where $\frq=\mathfrak m_V\cap C$.
	
	Set
	$
	B=R_0[b_1,\dots,b_r]\subseteq \Quot(R)$ so that $W^{-1}B=C$. Let 
	$
	 \frp=\frq\cap B.
	$
	Since every element of $W$ is a unit of $R$, one has $\frp\cap W=\varnothing$. Therefore, 
	$ \frq = \frp C
	$ 
	 and 
	$
	B_{\frp}= C_{\frq}=V
	$. 
	As $V$ is one-dimensional, we  have
	$
	\height(\frp)=1
	$, finishing the proof. 
\end{proof}

		
	\section{Proof of \autoref{thm:symbolic-main}}	
	
	By \autoref{cor:symbolic-saturation}, there exists  $h\in R$ such that
	$
	I^{(n)}=I^n:h^\infty
	$ 
	for every $n\ge 1$. 
	Write
	\[
	R=W^{-1}R_0,
	\qquad R_0=A/H,
	\qquad A=\kk[x_1,\dots,x_t],
	\]
	where $\kk$ is a field,  $R_0$ is a finitely generated $\kk$-algebra,  and $W\subseteq R_0$ is a multiplicatively closed set. Choose generators of $I$ and clear their denominators to assume $I=JR$ for some  ideal
     $
	J=(g_1,\dots,g_r)\subseteq R_0.
	$ Write $h=a/w$ with $a\in R_0$ and $w\in W$. Since $w$ is a unit in $R$ and saturations commute with  localization,  we have  	
	\begin{equation}\label{eq:saturationsR}
		I^{(n)}=I^n:_Ra^\infty= (J^n:_{R_0}a^\infty)R\quad \text{for every }n\ge 0.	
	\end{equation}

	Choose lifts $\tilde{g}_1,\dots,\tilde{g}_r,\tilde a\in A$ of $g_1,\dots,g_r,a$. Let
     $
	\widetilde J=(\tilde{g}_1,\dots,\tilde{g}_r)\subseteq A
	$ so that $J^n = (\widetilde J^n+H)R_0$. 
	By \autoref{lem:saturation-elimination}, if $z$ is a new variable and
	\[
	K_n=(\widetilde{J}^n+H,1-z\tilde a)\subseteq A[z],
	\]
	then
	$
	K_n\cap A=(\widetilde J^n+H):\tilde a^\infty
	$. Thus 
		\begin{equation}\label{eq:saturationsA}
		(K_n\cap A)R_0=J^n:_{R_0}a^\infty \quad \text{for every }n\ge 0.	
	\end{equation}

	If $\delta$ is the maximum total degree among the elements $\tilde g_j$, then $\widetilde J^n$ is generated by elements of total degree at most $n\delta$. Moreover, the generators of $H$ and the equation $1-z\tilde a$ are fixed. It follows that  there is a constant $c$ such that $K_n$ is generated by elements of total degree at most $cn$ for every $n\gs 1$. By \autoref{lem:effective-elimination} one has $\mu(K_n\cap A)=O(n^b)$ for some $b\in \ZZ_{>0}$. Thus, from \autoref{eq:saturationsA} and  \autoref{eq:saturationsR} we obtain $\mu(I^{(n)})=O(n^b)$, finishing the proof.
	\qed
	
	
	\section{Proof of \autoref{thm:div-main}}	
	Write $I_n
	=
	\bigcap_{i=1}^s I(v_i)_{n\lambda_i}$ with  $v_1,\ldots,v_s$ divisorial valuations of $R$ and $\lambda_1,\ldots,\lambda_s$  positive real numbers. 
 Let $V_i$ be the valuation ring of $v_i$, and normalize   $v_i$ and $\lambda_i$ to assume, without loss of generality, that $v_i=\ord_{V_i}$. Set 
	$m_i(n)=\lceil n\lambda_i\rceil$. 
	Then
    \[
	I(v_i)_{n\lambda_i}=\{f\in R\mid v_i(f)\ge m_i(n)\}.
	\]
	
	Write
	\[
	R=W^{-1}R_0,
	\qquad R_0=A/H,
	\qquad A=\kk[x_1,\dots,x_t],
	\]
	where $\kk$ is a field,  $R_0$ is a  domain that is finitely generated over $\kk$,  and $W\subseteq R_0$ is a multiplicatively closed set. By \autoref{lem:affine-dvr-model}, for each $i$ there exists a domain
	$
	R_0\subseteq B_i\subseteq \Quot(R)
	$ 
	finitely generated as an $R_0$-algebra,  and a height-one prime $\frp_i\subset B_i$ such that
	$
	(B_i)_{\frp_i}=V_i
	$. Therefore, for every $m\in \ZZ_{\gs 0}$ one has
	\begin{equation*}
	\{f\in R_0\mid v_i(f)\ge m\}
	=R_0\cap \frp_i^m(B_i)_{\frp_i}
	=R_0\cap \frp_i^{(m)}.
	\end{equation*}
	Since every $w\in W$ is a unit in  $R\subseteq V_i$, we have  $v_i(w)=0$, and so localization at $W$ does not change the values of $v_i$. Thus,  if
	$
	J_{i,m}=R_0\cap \frp_i^{(m)},
    $
	then
	\[
	J_{i,m}R=\{f\in R\mid v_i(f)\ge m\}. 
	\]
	Hence, if
	$
	J_n=\bigcap_{i=1}^sJ_{i,m_i(n)}
	$, then 
    $
	I_n=J_nR
	$.
	
	By \autoref{lem:uniform-saturation}, for each $i$ there exists an element
	$
	h_i\in B_i\setminus \frp_i
	$
	such that
	$
	\frp_i^{(m)}=\frp_i^m:h_i^\infty
	$
	for every $m\gs 1$. 
	Choose a presentation
	\[
	A_i=A[\mathbf y_i]=A[y_{i,1},\dots,y_{i,r_i}]\twoheadrightarrow B_i
	\]
	with kernel $H_i$. Choose lifts
	$p_{i,1},\dots,p_{i,a_i}\in A_i$
	of a set of generators of $\frp_i$, and a lift $\tilde h_i\in A_i$ of $h_i$. Let
	$
	Q_i=(p_{i,1},\dots,p_{i,a_i})\subseteq A_i.
	$
	For each $i$ introduce a new variable $u_i$, and set
	\[
	K_{i,n}
	=\bigl(Q_i^{m_i(n)}+H_i,\,1-u_i\tilde h_i\bigr)
	\subseteq A_i[u_i]
	\]
	for every $n\ge 1$. By \autoref{lem:saturation-elimination}, we have 
    $
	K_{i,n}\cap A_i=\bigl(Q_i^{m_i(n)}+H_i\bigr):\tilde h_i^\infty
	$. Thus
	\begin{equation}\label{eq:saturationsAi}
		(K_{i,n}\cap A_i)B_i = \frp_i^{m_i(n)}:h_i^\infty = \frp_i^{(m_i(n))},
	\end{equation}
	and so 
	$
	K_{i,n}\cap A
	=\pi^{-1}\bigl(R_0\cap \frp_i^{(m_i(n))}\bigr)
	=\pi^{-1}(J_{i,m_i(n)})
	$, 
	where $\pi:A\twoheadrightarrow R_0$.
	
	Let $z_1,\ldots, z_s$ be new variables, and  gather all the variables into the polynomial ring
	\[
	T=A[\mathbf y_1,\dots,\mathbf y_s,u_1,\dots,u_s,z_1,\dots,z_s]. 
	\]
	Extend each $K_{i,n}$ to $T$ and define
	\[
	D_n
	=\sum_{i=1}^sz_iK_{i,n}T
	+\Bigl(1-\sum_{i=1}^sz_i\Bigr)T.
	\]
	\autoref{lem:intersection-elimination} gives
	\begin{equation}\label{eq:Dn}
	D_n\cap A
	=\bigcap_{i=1}^s(K_{i,n}\cap A)
	=\pi^{-1}(J_n).
	\end{equation}
	
	The powers $Q_i^{m_i(n)}$ have  generating sets consisting of products of $m_i(n)$ of the  polynomials $p_{i,j}$. Moreover, the generators of the $H_i$ and  the equations $1-u_i\tilde h_i$   are fixed. Since
	$
	m_i(n)=\lceil n\lambda_i\rceil=O(n)
    $, we obtain that each ideal $K_{i,n}$ has a generating set whose maximum total degree is $O(n)$.   It follows that  there is a constant $c$ such that $D_n$ is generated by elements of total degree at most $cn$ for every $n\gs 1$. 
	By \autoref{lem:effective-elimination} and \autoref{eq:Dn}, $\mu(\pi^{-1}(J_n))=O(n^b)$ for some $b\in \ZZ_{>0}$. The result now follows from the equalities $I_n=J_nR=(\pi^{-1}(J_n))R$.
	\qed
		
		\bibliographystyle{plain}

		\bibliography{References}
		
		
	\end{document}